\documentclass[11pt]{amsart}
\usepackage{amssymb, amsmath, amsthm}
\usepackage{amsrefs}
\usepackage{graphicx}
\usepackage[final]{hyperref}
\usepackage[a4paper, centering]{geometry}

\usepackage{color}
\usepackage{graphicx}

\newtheorem{theorem}{Theorem}
\newtheorem{proposition}[theorem]{Proposition}

\theoremstyle{definition}
\newtheorem{definition}[theorem]{Definition}
\theoremstyle{remark}
\newtheorem{remark}[theorem]{Remark}

\def\R{\mathbb{R}}

\def\N{\mathbb{N}}

\def\haus{\mathcal{H}^{n-1}}

\def\l{\lambda}
\newcommand{\res}{\mathop{\hbox{\vrule height 7pt width .5pt depth 0pt
\vrule height .5pt width 6pt depth 0pt}}\nolimits} 

\def\curv{G}
\newcommand{\surf}[1][K]{S_{#1}}
\newcommand{\reg}[1][K]{R(#1)}

\def\tor{\mathcal{T}}
\def\Cap{{\rm cap}}
\def\costa{a}
\def\costb{b}
\def\g{\gamma}
\def\tc{\widetilde C} 
\def\tu{\widetilde u}
\def\th{\widetilde h}
\def\tk{\widetilde G}
\def\tn{\widetilde \nu}
\def\ttor{\widetilde{ F}}

\begin{document}

\title[]%
{Variational worn stones}

\author[G.~Crasta, I.~Fragal\`a]{Graziano Crasta,  Ilaria Fragal\`a}

\address[Graziano Crasta]{Dipartimento di Matematica ``G.\ Castelnuovo'',
Sapienza University of Rome\\
P.le A.\ Moro 5 -- 00185 Roma (Italy)}
\email{graziano.crasta@uniroma1.it}

\address[Ilaria Fragal\`a]{
Dipartimento di Matematica, Politecnico\\
Piazza Leonardo da Vinci, 32 --20133 Milano (Italy)
}
\email{ilaria.fragala@polimi.it}



\date{\today}

\begin{abstract}  We introduce an evolution model \`a la Firey for a convex stone which tumbles on a beach and undertakes an erosion process depending on some variational energy, such as torsional rigidity, principal Dirichlet Laplacian eigenvalue, or Newtonian ca\-pa\-ci\-ty. 
Relying on the assumption of existence of a solution to the corresponding parabolic flow, we prove that the stone tends to  become asymptotically spherical.  
Indeed, we identify an ultimate shape of these flows with a smooth convex body whose  ground state 
satisfies an additional  boundary condition,
and we prove  symmetry results for the  corresponding overdetermined elliptic problems. Moreover, we extend the analysis to arbitrary convex bodies: we introduce  
new notions of cone variational measures and 
we prove that, if such a measure is absolutely continuous with constant density,  the underlying body is a ball. 
\end{abstract} 
\maketitle

\section{Introduction}\label{sec:intro} 
Aim of this paper is to propose a variational counterpart of Firey's  seminal  result stating that 
the fate of worn stones is that of becoming spherical   \cite{firey}.
Firey considered the evolution problem satisfied by the support functions  $h ( t ,\cdot)$ of a family of convex bodies $C (t)$ in $\R ^3$, 
obtained when an initial convex stone $C(0)$ tumbles on an abrasive plane  and  undertakes an erosion process, in which the rate of wear is proportional to the density of contact points with the  plane per unit surface area, and also to the volume of the stone.  The mathematical formulation reads
\begin{equation}
\label{f:fireyevolution}
\begin{cases}
 \displaystyle  \frac{ \partial h}{\partial t} ( t, \xi) =   - \costa|C(t)|   { \curv (t, \nu _t^ { -1} ( \xi ) )}  & \text{ on } (0, T) \times \mathbb S ^ { n-1} ,\\
h ( 0, \xi) = h _ 0 ( \xi)  & \text{on} \ \mathbb S ^ {n-1},\\
h ( t, \xi) \leq \costb& \text{on} \  (0, T) \times \mathbb S ^ {n-1} \,,
\end{cases}
\end{equation}
where $\costa$ and $\costb$ are positive constants,  $\nu _t^ { -1}$ is the inverse Gauss map of $C (t)$,  $\curv (t, \cdot)$ denotes the Gaussian curvature of $C (t)$, and $| C(t)|$ its  volume.  
Under the assumption that $C (0)$ is smooth and centrally symmetric, and that, for some $T>0$, the evolution problem~\eqref{f:fireyevolution}
admits a smooth solution, Firey showed that actually $T = + \infty$, that convexity and symmetry are preserved along the flow, and that an ultimate shape  must satisfy the  following geometric condition:  in every direction 
its support function $h$ and Gaussian curvature $G$ are  proportional to each other. Precisely,  denoting by $\nu$ its Gauss map 
and by $c$ a positive constant,  it  satisfies 
\begin{equation}\label{f:overfirey}
h ( \xi ) = c \, \curv (\nu ^ { -1} (\xi)) \qquad \forall \xi \in \mathbb S ^ 2\,.
\end{equation}
 Then Firey obtained the following symmetry result:  
 if a smooth and centrally symmetric convex body in $\R ^3$ satisfies the equality \eqref{f:overfirey}, it must be  a ball. 
 Moreover he conjectured that, in dimension $3$, the above results (often referred to as ``convergence of solutions to a round point'') 
should remain valid without the assumption of central symmetry; 
in higher dimensions, this is also called the generalized Firey conjecture.

 Firey's vision of worn stones has inspired many deep developments up today, in both the fields of evolution equations, and of convex geometry. 

On the side of evolution equations, the parabolic problem introduced by Firey, written in terms of a parametrization of the boundary and up to a renormalization, is nothing else than the
Gaussian curvature flow. Such flow has been intensively studied in recent years, so that it is impossible to report the related literature. 
We refer to the recent monograph \cite{ACGL} for an account about the Gaussian curvature flow and more references on it. 
Let us just mention that nowadays  a full resolution is available for both Firey's conjecture and its generalized version, 
the main steps of its history being the following: the existence of a solution to the Gauss curvature flow was proved in any space dimension 
by Chou \cite{Tso};   the convergence to a round point
was obtained  in dimension  $n = 3$  by Andrews   \cite{A99},
while in higher dimensions  it follows from a convergence result  by 
Guan-Ni \cite{GuanNi}, combined with a rigidity result  by Brendle-Choi-Daskalopoulos \cite{BCD}. 

On the side of convex geometry, the  rigidity result by Firey  is related to more general rigidity questions for convex bodies. 
In fact,  
condition   \eqref{f:overfirey} amounts to ask that a convex body $K$ has the same {\it cone volume measure} $V_K$ as a ball (see Section \ref{sec:prel} for the definition).
Accordingly, a generalized version of Firey rigidity problem reads as follows:
  for  $K$ belonging to the class $\mathcal K ^n$ of convex bodies in $\R ^n$, and a ball $B$,  investigate the validity of the implication  \begin{equation}\label{f:critical00} V _{ K } = \, V _ B   \quad \Rightarrow\quad K  = B \,.
 \end{equation}  
In the above mentioned papers about the asymptotic behaviour of the Gaussian curvature flow, it has been proved that \eqref{f:critical00}  holds under the assumption that $K$ has its centroid at the origin \cite[Proposition 3.3]{GuanNi} 
or that $K$ is smooth and strictly convex \cite[Theorem 6]{BCD}.

Looking  at \eqref{f:critical00},
one is led in a natural way to the more general  question whether the cone volume measure  determines uniquely the associated convex body. 
In particular, denoting by $ \mathcal K ^n _*$ the class of centrally symmetric convex bodies in $\R^n$,  the question is whether, 
given $K$ and $K_0$ in $\mathcal K ^n_*$,  it holds that
 \begin{equation}\label{f:critical0}  V _{ K} = V _ {K_0}   \quad \Rightarrow\quad K = K _0  \,.
 \end{equation}  
The relevance of this question lies in particular in its connection with one of the major open problems  in convex geometry, namely the validity of the
 log-Brunn-Minkowski inequality: 
\begin{equation}\label{f:logBM}  \big  | (  1-\lambda) \cdot K +_0 \lambda \cdot  L  \big |  \geq  |K|^ { 1- \l}  |L |^ \l  \qquad \forall K, L \in \mathcal K ^ n _*\,, \ \forall \l \in [0, 1]\,,
\end{equation}
where 
\[
( 1-\lambda) \cdot K +_0 \lambda \cdot  L := 
\Big \{ x \in \R ^n \ :\ x \cdot \xi \leq h _ K ( \xi) ^ { 1- \l} h _ L (\xi) ^ {\l} \quad \forall \xi \in  \mathbb S ^ { n-1} \Big \}\,,
\]
and $h_K$, $h_L$ denote the support function of $K$ and $L$ respectively.

Inequality \eqref{f:logBM}, which is a strengthening of the classical Brunn-Minkowsi inequality,  has been proved  in dimension $n=2$ in \cite{BLYZ}, while it is currently open in higher dimensions (see the recent  paper \cite{Bor0} for an extensive survey about the state of the art in this topic, and also our previous paper \cite{CF12} for a related functional version). 
As shown in \cite{BLYZ}, proving the log-Brunn-Minkowski inequality is equivalent to proving that,
 for every fixed $K \in \mathcal K ^n_*$, the minimization problem
 \begin{equation}\label{f:minpb0} \min \Big \{ \frac{1}{|K | } \int_{\mathbb S ^ {n-1}}  \!\!\! \log   (h _L  )    d  V _K \ : \ L \in \mathcal K ^ n _* \, , \  |K|=  |L| \Big \}\end{equation}
 is solved by $K$ itself.   Moreover, it turns out that a solution  $K_0$ to problem \eqref{f:minpb0} exists and is a critical set, in the sense that it satisfies the equality of measures $ V _{ K _0} =  V _ K$. Hence the connection with the implication \eqref{f:critical0}. For $n = 2$, 
 B\"or\"oczky-Lutwak-Yang-Zhang \cite{BLYZ} proved that such implication holds true, thus obtaining also the log-Brunn-Minkowski inequality. 

To close the circle, the minimization problem \eqref{f:minpb0} is not unrelated with the evolution of convex bodies by their Gaussian curvature.   
Indeed, in the particular case $K = B$, the integral in \eqref{f:minpb0} is  an entropy functional, firstly considered by Firey,
which decreases along the Gaussian curvature flow (and actually a suitable generalization of this entropy monotonicity property
is used by Guan-Ni to obtain  \eqref{f:critical00}  for bodies with centroid at the origin).

\medskip
In this work we attack the new problem of studying the above topics  when the volume functional is replaced by some variational energy, such as torsional rigidity, principal Dirichlet Laplacian eigenvalue, or Newtonian capacity.  These functionals are object of study in many problems of classical and modern Calculus of Variations, such as isoperimetric type inequalities, concavity inequalities, inequalities involving polarity. In many cases,  the behaviour of these functionals 
turn out to resemble that  of the volume functional, in particular concerning the validity of inequalities of Brunn-Minkowski type (see \cite{Col2005}). 

Thus, studying variational results \`a la Firey  seems to be a  very natural direction, which is, to the best of our knowledge, completely unexplored. 
To precise what we intend, let us  focus our attention in particular on the case of torsional rigidity. 
In the physical case when $K \subset \R ^2$ is the cross section of a cylindrical rod $K\times \R$ under torsion, its torsional rigidity $\tor (K)$ is
the torque required for unit angle of twist per unit length.  From an analytical point of view,  when 
 $K$ is more in general a $n$-dimensional convex body, 
 its torsional rigidity  (often abbreviated as torsion) is given by
$$ \tor (K) = \int_{K} u_K \, dx \,,$$ 
 where  $u_K $ is the torsion function of $K$, i.e.\ the unique solution to the Dirichlet problem
\begin{equation*}
\left\{
\begin{array}{ll}
-\Delta u=1 \ &\hbox{ in }  {\rm int}\, K  \\
u=0 \ & \hbox{ on  } \partial K\, .
\end{array}\right.
\end{equation*}

A classical result by Dahlberg \cite{Dahl} ensures  that, for any $K \in \mathcal K ^n$, the gradient of  the torsion function  is well-defined $\mathcal H ^ {n-1}$-a.e. 
on $\partial K$, and belongs to $L ^ 2 (\partial K)$.    
Moreover, it is also well-known that  the Hadamard first variation of torsion
can be expressed through an integral formula,
which is the perfect analogue of the one valid for volume,  
replacing the surface area measure $S_K$  by the {\it torsion first variation measure}    $\mu _K$,   defined by 
\begin{equation}\label{f:firstvar} 
 \mu _K := (\nu _K) _\sharp \big (  |\nabla u_K |^ 2  {\mathcal H} ^ {n-1} \res \partial K   \big ) \,,
\end{equation}  
$(\nu _ K )\sharp$ denotes the push forward through the Gauss map of $K$.
By pursuing the analogy with the case of  volume, we are led to introduce  the  {\it cone torsion measure} of $K$,  as the positive measure on ${\mathbb S} ^ {n-1}$ defined by
\begin{equation*}
\tau _K := h _K \,  {\mu _K} \,,
\end{equation*} 
where  $h _K$ denotes the support function of $K$.

Relying on this new definition, it is natural to investigate the rigidity question analogue to the one of Firey, namely whether, 
for a given $K \in \mathcal K ^n _*$, 
 \begin{equation*}
\tau _{ K } =  \tau _ B   \quad \Rightarrow\quad K = B  \,.
 \end{equation*}  
In the setting of smooth convex bodies, the equality $\tau _{ K } =  \tau _ B$  amounts to ask that the torsion function  $u _ K$  satisfies, for some positive constant $c$, 
the  boundary condition 
\begin{equation*}
|\nabla u _K|^2  {x\cdot \nu _K } = c  \, {\curv }_K \qquad \text{ on } \partial K \,, 
\end{equation*} 
where $\curv_K$ denotes the Gaussian curvature of $\partial K$. 
Hence, the corresponding rigidity problem consists in investigating symmetry of smooth convex bodies $K$  such that 
the following overdetermined boundary value problem  admits a solution: 
 \begin{equation}
\label{f:prob10}
\begin{cases}
-\Delta u = 1 & \text{in}\  {\rm int} \, K ,\\
u = 0 & \text{on}\ \partial K,\\
|\nabla u|^2 \displaystyle{ {x\cdot \nu_K }}= c \, {\curv}_K  
& \text{on}\ \partial K\,.
\end{cases}
\end{equation}
This problem is not covered by the very vast literature on Serrin-type problems \cite{Se}, as the simultaneous presence of the support function and of curvature in the overdetermined boundary condition is completely new; we limit ourselves to quote the papers 
\cite{BGNT, f} about overdetermined problems in which curvatures are involved.

In Theorem \ref{t:reg1} we establish symmetry for problem \eqref{f:prob10}, under the weaker assumption that $K$ has its centroid at the origin. At present, we do not know if the assumption on the centroid can be removed. The proof of Theorem \ref{t:reg1}  is obtained by an ad-hoc combination of different inequalities, such as the Saint-Venant inequality for torsion, and the isoperimetric inequality for the $2$-affine surface area (or alternatively, an isoperimetric-type inequality proved in \cite{BFL12} and Blasckhe-Santal\'o inequality). 

In the setting of arbitrary  convex bodies, 
 the equality $\tau _{ K } =  \tau _ B$ 
 cannot be formulated any longer as a pointwise equality for  the gradient of the torsion function along the boundary; it just   
 tells  the measure $\tau _K$ is a constant multiple of the diffuse measure on the sphere, i.e. 
 \begin{equation}\label{f:weak}
\tau _K= c \, \mathcal H ^ {n-1} \res { \mathbb S} ^ { n-1}  \,.
\end{equation} 
In Theorem \ref{t:nonreg1} we establish symmetry for  convex bodies with centroid at the origin  satisfying \eqref{f:weak}. We argue by adapting the arguments used in the regular case, but the proof is more delicate, as it involves 
the  regularity properties of convex bodies having an absolutely continuous  surface area measure; in this respect,  we heavily rely on the results proved by Hug in the papers \cite{Hugc, Hug1, Hug2}. 

The symmetry results described so far are related to many further questions. 
In particular, in the spirit of Firey's work,  it is natural to wonder whether 
the condition $\tau _K= \tau _ B$  identifies $K$ as the ultimate shape of some  flow.

To answer this question, we imagine a  new variational flow for worn stones which tumble on an abrasive plane and undertake an energy-based erosion process. Namely, 
we assume that the rate of wear is  proportional to the density   of contact points with the abrasive plane, 
 no longer per unit surface area measure as in \cite{firey},
but per unit torsion first variation measure;  moreover, while in \cite{firey} the rate of wear is also taken to be proportional to the volume, we take it to be proportional to the torsional rigidity.

To write explicitly the problem, let $C (t)$ represent the evolution in time of an initial convex stone $C (0)$, and let $\sigma$ be any small part of  $\partial C (t)$. 
 The measure of the set of directions for which the abrasive plane touches $C(t)$ at points in $\sigma$ is given by $\mathcal H ^ {n-1} (\nu _t (\sigma))$, where $\nu _ t$ is the Gauss map of $C (t)$.
 By a standard change of variables (see Proposition \ref{p:Hug} (e)), we have 
   $$\mathcal H ^ {n-1} (\nu_t (\sigma)) = \int _{\sigma} \curv(t, y) \, d \mathcal H ^ {n-1}(y) \,,$$ 
  where $\curv(t, y)$ denotes the Gauss curvature of $\partial C ( t) $ at the point $y$.    
   Recalling the definitions of surface area measure and of torsion first variation measure for the convex body $C(t)$,  the above integral can also be written as 
   $$  \int _{\nu_t (\sigma)} \curv  (t , \nu _t ^ {-1} (\xi)) \, d S _{C(t)}  \qquad \text{ or } \qquad   \int _{\nu_t (\sigma)}  \frac{ \curv (t, \nu_t ^ { -1} ( \xi ) )}{|\nabla u _t( \nu_t ^ { -1}  (\xi)) | ^ 2}  \, d \mu  _{C(t)}   \,,$$ 
   where $u (t, \cdot)$ is the torsion function of $C (t)$. 
   Thus we see that, while   contact points have density $\curv  (t , \nu _t ^ {-1} (\xi))$ with respect to the surface area measure $S_{C ( t)}$, they have density 
$ \frac{ \curv (t, \nu_t ^ { -1} ( \xi ) )}{|\nabla u( t, \nu_t ^ { -1}  (\xi)) | ^ 2} $ with respect to the first variation measure $\mu _ { C ( t) }$ in \eqref{f:firstvar}.
 We conclude that the variational analogue of Firey's problem when replacing volume by torsion reads as follows: 
\begin{equation}
\label{f:evolution}
\begin{cases}
 \displaystyle  \frac{ \partial h}{\partial t} ( t, \xi) =   - \costa\tor (C( t))  \frac{ \curv (t, \nu_t  ^ { -1} ( \xi ) )}{|\nabla u( t, \nu_t ^ { -1}  (\xi)) | ^ 2}  & \text{ on }  (0, T) \times {\mathbb S} ^ { n-1} ,\\
h ( 0, \xi) = h _ 0 ( \xi)  & \text{on} \ {\mathbb S} ^ {n-1},\\
h ( t, \xi) \leq \costb& \text{on} \ (0, T) \times {\mathbb S} ^ {n-1}\,.
\end{cases}
\end{equation}
As in Firey's paper, we assume that 
$C (0)$ is smooth and centrally symmetric, and that for some $T>0$, the evolution problem  \eqref{f:evolution}
admits a smooth solution. Under this assumption, 
Theorem \ref{t:firey} establishes that $T = + \infty$,  
that 
convexity and symmetry are preserved along the flow, and that an ultimate shape has the same cone torsion measure as a ball. The proof relies basically on the Brunn-Minkowski inequality for torsion due to Borell \cite{bor1}. 

We remark that proving the existence of a unique smooth solution to  problem \eqref{f:evolution} is an interesting problem in parabolic flows, which is beyond the scopes of this paper, and will be object of further research.

To conclude, another question we wish to address is whether, by analogy with \eqref{f:critical0}, for any pair of convex bodies $ K, K _0 \in \mathcal K ^n _*$,  we have \begin{equation}\label{f:tau0}  \tau _{ K} = \tau _ {K_0}   \quad \Rightarrow\quad K = K _0  \,.
 \end{equation}  
Such implication is in turn related  to the possibility of strengthening the Brunn-Minkowski inequality for torsion 
into  a  logarithmic Brunn-Minkowski inequality of the kind  \begin{equation}\label{f:logBMtau} \tor \big  ((  1-\lambda) \cdot K +_0 \lambda \cdot  L  \big ) \geq \tor(K) ^ { 1- \l} \tor(L) ^ \l  \qquad \forall K, L \in \mathcal K ^ n _*\,, \ \forall \l \in [0, 1].
\end{equation}
 Indeed, by arguing as in the  proof of Lemma 3.2  in \cite{BLYZ},  one can easily check that  \eqref{f:logBMtau} is equivalent to proving that, 
 for every fixed $K \in \mathcal K ^n_*$, the minimization problem
 \begin{equation}\label{f:minpbt0} \min \Big \{ \frac{1}{\tor (K) } \int_{\mathbb S ^ {n-1}}  \!\!\! \log   (h _L  )    d  \tau _K \ : \ L \in \mathcal K ^ n _* \, , \  \tor (K)=  \tor (L) \Big \}\end{equation}
 is solved by $K$ itself;   moreover, if the above problem admits   solution $K_0$, it  must satisfy  the first order optimality condition  $ \tau _{ K _0} =  \tau _ K$ (cf.\ the proof of Theorem 7.1 in \cite{BLYZ}). 

The inequality \eqref{f:logBMtau}  (or the implication \eqref{f:tau0}) is likely very challenging, and  it should be first proved (or disproved)   in dimension $n = 2$, see Section \ref{sec:final} for some related remarks. 

To conclude,  we point out that 
all the results discussed so far  are valid also when  torsional rigidity is replaced
either by the principal frequency of the Dirichlet Laplacian, or (for $n \geq 3$) by  the Newtonian capacity.
To some extent, this may lead to consider the analysis our variational worn stones more righteous, and to believe that the corresponding parabolic flows and logarithmic inequalities are worth of further investigation.
  
\bigskip
{\bf Outline of the paper.}  In Section \ref{sec:prel} we detail the definitions of cone variational measures  and  related representation formulas. 
In Section  \ref{sec:overdet} we prove symmetry for smooth centred convex bodies $K$ 
where
overdetermined boundary value problems such as \eqref{f:prob10} 
admit a solution.  In Section \ref{sec:cone} we 
 extend the rigidity results of the previous section to the general framework of arbitrary centred convex bodies.  
 In Section \ref{sec:flow} we relate  the convex bodies studied in the previous sections to new variational flows. 
 In Section \ref{sec:final} we give a few concluding remarks about logarithmic
Brunn-Minkowski-type inequalities. \bigskip 
\section{Cone variational measures and related representation formulas}\label{sec:prel} 

\medskip
Let $\mathcal K ^n$ be the class of convex bodies  in $\R^n$ with the origin in their interior.  In particular, we are going to consider convex bodies 
 in $\mathcal K ^n$ whose centroid $${\rm cen} (K) = \frac{1}{|K|} \int_K  x \, dx$$
 coincides with the origin.  Here and below, we indicate by $| \cdot |$ the volume functional on $\mathcal K ^n$. 
 We shall also consider the subclass  $\mathcal K ^ n _*$  of centrally symmetric convex bodies. 
 
 Given $K \in \mathcal K ^n$,  we denote by $h _ K$ and $\nu _K$ respectively its support function and Gauss map. 
By definition, if $m$ is a measure on $\partial K$,  its push-forward $( \nu _K )_ \sharp(  m) $ through the Gauss map is the measure on $\mathbb S ^ {n-1} $ defined by 
\[
\int_{\mathbb S ^ {n-1}} \varphi  \, d ( \nu _K )_ \sharp (m )  = \int _{\partial K} \varphi  \circ \nu _K   \, dm  
\qquad \forall \varphi \in \mathcal C (\mathbb S ^ {n-1}) \,.
\] 

We recall that the {\it surface area measure} $S_K$ and the {\it cone volume measure} $V_K$ are  defined respectively by 
$$S_K = ( \nu _K )_ \sharp (  {\mathcal H} ^ {n-1} \res \partial K    ) \, , \qquad V_K = h _ K  S_K\,;
$$ 
moreover, denoting by $|V_K|$ the total variation of $V_K$,  it holds that
$$\begin{array}{ll}
& \displaystyle| K|  = \frac{1}{ n} \int _{ \mathbb S ^ {n-1} } h _ K \, d S _K  = \frac{1}{n} |V_K| \,,
\\ \noalign{\bigskip} 
& \displaystyle \frac{d}{dt} |  K + t L | \Big |_{ t = 0 ^ + }   =   \int _{ \mathbb S ^ {n-1} } h _ L \, d S_K \,.
\end{array}$$ 
The name cone volume measure is motivated by the fact that, when $K$  is a polytope with facets $F_i$ and unit outer normals $\nu _i$, it holds that
\begin{equation}\label{f:polytopes} V_K =  \sum _i  |\Delta (o, F_i ) |  \delta _{\nu _i }\,,  \end{equation} 
where $\delta _{\nu _i}$ is a Dirac mass at $\nu _i$ and $\Delta (o, F _i)$ is the  cone with apex $o$ and basis $F_i$. 
In recent years, this notion of cone volume measure has been widely studied (see  \cite{BH15, BoroHenk, BGMN, BLYZ, BLYZ2, GM87,ludwig, LR10, LYZ05, naor1, naor2, PW12, stancu12, zhu1, zhu2}). 

Let now  $F (K)$ be one of the following variational energies: torsional rigidity $\tor (K)$, first Dirichlet Laplacian eigenvalue $\lambda _ 1 ( K)$, or Newtonian capacity $\Cap (K)$ (the latter in dimension $n \geq 3$).  We denote by $u _K$ the corresponding  ground state, defined as the unique solution in $\Omega= {\rm int} K$ to the following 
elliptic boundary value problems (in the second case normalized so to have unit $L ^2$-norm):
$$
\begin{cases} 
-\Delta u=1  &\hbox{ in }  \Omega  \\
u=0 \ & \hbox{ on  } \partial \Omega
\end{cases}
\qquad
\begin{cases}
-\Delta u = \lambda_1(\Omega) u & \text{in}\ \Omega,\\
u = 0 & \text{on}\ \partial \Omega
\end{cases}
 \qquad 
\left\{
\begin{array}{ll}
\Delta u= 0 \ &\hbox{ in } \R^n \setminus K \\
\noalign{\smallskip} 
u=1 \ & \hbox{ on  } \partial K \\
\noalign{\smallskip} 
 u (x)  \to 0 & \hbox{ as  } |x| \to + \infty \,. \\
\end{array}\right.
$$

It is well known that, by analogy with the case of volume, each of these functionals and its Hadamard derivative satisfy the following representation formulas (see respectively \cite{CoFi},  \cite{Je} and \cite{Je96})
$$\begin{array}{ll}
& \displaystyle F ( K) = \frac{1}{|\alpha| } \int _{ \mathbb S ^ {n-1} } h _ K \, d \mu _K 
\\ \noalign{\bigskip} 
& \displaystyle \frac{d}{dt} F ( K + t L ) \Big |_{ t = 0 ^ + }   = {\rm sign} (\alpha)  \int _{ \mathbb S ^ {n-1} } h _ L \, d \mu _K \,.
\end{array}$$ 
Here $\alpha$ denotes 
the homogeneity degree of $F$ under domain dilation (which equals 
equal to $n + 2$ for torsion,  $-2$ for the eigenvalue, and $n -2$ for capacity) and 
  $\mu _K$ is the {\it first variation measure} of $F$, defined by 
\begin{equation}\label{f:mufir} 
\mu _K :=(\nu_K ) _\sharp \big (  |\nabla u_K |^ 2  {\mathcal H} ^ {n-1} \res \partial K   \big )
\end{equation}

We set the following: 

\begin{definition}[cone variational measures] \label{def:cone} 
For any of the functionals above, with any $K \in \mathcal K ^n$ we associate the  positive measure  on $\mathbb S ^ {n-1}$ defined by 
$h _ K \, \mu _K$, being $\mu _K$ given by  \eqref{f:mufir}, and we 
denote it  respectively by 
 $\tau _K$ ({\it cone torsion measure}), by 
  $\sigma _K$ ({\it cone eigenvalue measure}) and by 
  $\eta _K$  ({\it cone capacitary measure}).  
\end{definition}
Let us point out that the terminology in Definition \ref{def:cone} is chosen to emphasize the analogy with the case of volume, though this is 
somehow an abuse; indeed, the analogue of the equality \eqref{f:polytopes}  is clearly false for our variational funtionals, since they
are not additive under the decomposition of a polytope  in $\mathcal K ^n$ into cones with apex $o$ and bases at its facets. 

Let us also notice that  the above representation formulas  can be rewitten as integrals over the boundary and 
 in terms of the total variations of the corresponding cone variational measure, respectively as 
\begin{equation} \label{f:poho}
\tor(K) = \frac{1}{n+2} \int_{\partial K} |\nabla u _K|^2\, x\cdot \nu _K\, d\haus 
= \frac{1}{n+2} 
| \tau _K | \,,
\end{equation} 
\begin{equation} \label{f:pohoL}
\lambda _1 (K) = \frac{1}{2} \int_{\partial K} |\nabla u_K|^2\, x\cdot \nu_K\, d\haus 
= \frac{1}{2} 
| \sigma _K|   \,,
\end{equation}
\begin{equation} \label{f:pohoC}
\Cap (K) = \frac{1}{n-2} \int_{\partial K} |\nabla u_K|^2\, x\cdot \nu_K\, d\haus
= \frac{1}{n-2} | \eta _K |  \qquad (n \geq 3)\,.
\end{equation}

\begin{remark}\label{r:weak} 
The cone variational measures introduced in Definition \ref{def:cone} are weakly$^*$ continuous with respect to the convergence of convex bodies in Hausdorff distance. Indeed, if $K_n$ converge to $K_\infty$ in Hausdorff distance, the support functions of $K _n$ converge uniformly to the support function of $K_\infty$, 
while the  first variation measures of $K _n$ converge weakly$^*$ to the first variation measure of $K_\infty$ (cf.\ respectively \cite[Thm.~3.1]{Je96} for capacity,  \cite[Section~7]{Je} for the first eigenvalue, and \cite[Thm.~6]{CoFi} for torsion). 
\end{remark}

\section{Rigidity results for smooth convex bodies}\label{sec:overdet} 

In this section we deal with overdetermined boundary value problems on smooth convex bodies $K$, for which the 
Gauss map $\nu _K$ and the Gaussian curvature $G_K$ can be classically defined. 
Whenever no confusion may arise, we omit the index $K$ and we simply write $\nu$ and $G$. 
We denote by $B$ a generic ball, by $B _ 1$ the unit ball, and by $\omega _n$ its Lebesgue measure. 

\begin{theorem}\label{t:reg1} 
Let $K \in \mathcal K ^n$ have its centroid at the origin, and boundary of  class $\mathcal C^2$.
Assume that, for some constant $c>0$,  
 there exists a solution to the  following overdetermined boundary value problem on $\Omega := {\rm int } K$:
\begin{equation*}
\begin{cases}
-\Delta u = 1 & \text{in}\ \Omega,\\
u = 0 & \text{on}\ \partial\Omega,\\
|\nabla u|^2 \displaystyle{ {x\cdot \nu }}= c \, {\curv}  
& \text{on}\ \partial\Omega\,.
\end{cases}
\end{equation*}
Then $K$ is a ball.
\end{theorem}

The proof  of Theorem \ref{t:reg1} is based on 
the application of different inequalities for convex bodies,
some involving torsional rigidity and some others being purely geometrical. 
We list  them below:

\begin{itemize}

\item[(i)] The Saint-Venant inequality for torsion \cite{posz, H17}:
\[
\frac{\tor(K)}{|K|^{\frac{n+2}{n}}}
\leq
\frac{\tor(B)}{|B|^{\frac{n+2}{n}}}\,,
\]
with equality if and only if $K$ is a ball.
More explicitly, since $\tor (B_1) = \frac{\omega_n} {n(n + 2)}  $, 
\begin{equation}
\label{f:isotor}
\tor(K) \leq \frac{|K|^\frac{n+2}{n}}{n(n+2) \omega_n^{\frac{2}{n}}}\,,
\end{equation}
with equality if and only if $K$ is a ball.

\medskip
\item[(ii)] The isoperimetric inequality of the $p$-affine surface area in the case $p=2$:
\begin{equation}
\label{f:paffine}
\Theta_2(K) :=
\int_{\partial K}\frac{\curv^{\frac{2}{n+2}}}{(x\cdot\nu)^{\frac{n}{n+2}}}
\, d\haus
\leq
n \, \omega_n^\frac{4}{n+2}\, |K|^{\frac{n-2}{n+2}}\,,
\end{equation}
with equality if and only if $K$ is an ellipsoid
(see \cite[Theorem~4.8]{Lut96}, \cite[Theorem~4.2]{WerYe} or \cite[{formula (10.49)}]{Sch2}).

\medskip
\item[(iii)]  The isoperimetric-type inequality:
\begin{equation}\label{f:BFL1}
\frac{ \tor (K) ^ { 1 - \frac{1}{n+2}   }} {\int_{\partial K} |\nabla u_K| ^ 2 }
\leq 
\frac{ \tor (B) ^ { 1 - \frac{1}{n+2}  }  } {\int_{\partial B} |\nabla u _B| ^ 2 }
 \end{equation} 
 see \cite[Theorem 3.15]{BFL12}. 

\medskip
\item[(iv)] The Blaschke-Santal\'o inequality \cite{B23, S49} for convex bodies with centroid at the origin:
\begin{equation}\label{f:BS}
|K| |K ^o| \leq \omega _n ^ 2\,, 
\end{equation}
with equality if and only if $K$ is an ellipsoid (see \cite[Theorem $1'$]{Lut95}).
\end{itemize}

\bigskip 

{\it Proof of Theorem \ref{t:reg1}}. We are going to obtain 
an upper bound and a lower bound for the constant $c$ appearing in the overdetermined boundary condition, 
starting from two different  expressions of it. 
For the upper bound we use inequality (i), while the lower bound can be obtained by two different methods:
either by using the inequality (ii), or by using the inequalities (iii)-(iv).  
We provide both methods  because, as we shall see,  the first one works to extend the result to Newtonian capacity, while  the second one works for the first Laplacian Dirichlet eigenvalue. 
 Since the upper bound and the lower bound turn out to match each other, we conclude that in particular Saint-Venant inequality \eqref{f:isotor} holds as an equality, and hence $K$ must be a ball.

\bigskip
{\it -- Upper bound for $c$.} 
We integrate over the boundary both sides of the pointwise overdetermined condition  
$$|\nabla u|^2\, x\cdot \nu = c\, \curv \qquad \text{ on } \partial K\,.$$ 
By using respectively  the identity 
 \eqref{f:poho}  and the change of variables formula (2.5.29) in \cite{Sch},   we obtain
\[
\int_{\partial K} |\nabla u|^2\, x\cdot \nu \, d \mathcal H ^ {n-1} = (n + 2) \tor (K)\,,
\]
and  \begin{equation}\label{f:intk}
\int_{\partial K}  \curv\, d \mathcal H ^ {n-1} 
= \int _{\mathbb S ^ {n-1} } 1\, d \mathcal H ^ {n-1}  = \mathcal H ^ {n-1} (\mathbb S ^ {n-1}) = n \omega _n \,.
\end{equation} 
Hence we get  \begin{equation}
\label{f:c1}
c = \frac{n+2}{n\, \omega_n}\, \tor(K)\, .  
\end{equation}
Then by using the
Saint-Venant inequality \eqref{f:isotor}  we obtain the upper bound \begin{equation}
\label{f:above}
c \leq \frac{1}{n^2}\, \left(\frac{|K|}{\omega_n}\right)^{\frac{n+2}{n}}\,,
\end{equation}
with equality if and only if $K$ is a ball.

\bigskip
{\it -- Lower  bound for $c$ (first method).} 
We  integrate  
on $\partial K$ the
pointwise overdetermined condition, after rewriting it as  
\[
|\nabla u| = c^{1/2} \, \left(\frac{\curv}{x\cdot \nu}\right)^{1/2} \qquad \text{ on } \partial K\,.
\]
Exploiting the fact that $u$  solves the torsion problem,  we obtain \[
|K| = c^{1/2} \int_{\partial K} \left(\frac{\curv}{x\cdot \nu}\right)^{1/2}
\, d\haus\,,
\]
i.e.
\begin{equation}
\label{f:c2}
c = \frac{|K|^2}{\left[ \int_{\partial K} \left(\frac{\curv}{x\cdot \nu}\right)^{1/2}
\, d\haus\right]^2}\,.
\end{equation}
We now look at the integral appearing  the denominator in the right-hand side of \eqref{f:c2}.
Let us distinguish the cases $n= 2$ and $n \geq 3$. 

If $n = 2$, such integral is exactly  the $2$-affine surface area: 
$$ \int_{\partial K} \left(\frac{\curv}{x\cdot \nu}\right)^{1/2}
\, d\haus = 
 \Theta_2(K)\,.$$ 

If $n\geq 3$, the same integral
can be estimated in terms of the $2$-affine surface area using H\"older's inequality. Specifically, using H\"older's inequality with the conjugate
exponents $\beta = \frac{2n}{n+2}$ and $\beta' = \frac{2n}{n-2}$,
it holds that
\[
\begin{split}
\int_{\partial K} \left(\frac{\curv}{x\cdot \nu}\right)^{1/2}
\, d\haus
& =
\int_{\partial K} \frac{\curv^{1/n}}{(x\cdot \nu)^{1/2}}\,
\cdot \curv^{\frac{2-n}{2n}}
\, d\haus
\\ & \leq
\left[\int_{\partial K} \left(\frac{\curv^{1/n}}{(x\cdot \nu)^{1/2}}\right)^{\beta}
\, d\haus\right]^{1/\beta} \!\!
\cdot
\left[\int_{\partial K} \left(\curv^{\frac{2-n}{2n}}\right)^{\beta'}
\, d\haus\right]^{1/\beta'}
\\ & =
\left[\int_{\partial K}\frac{\curv^{\frac{2}{n+2}}}{(x\cdot\nu)^{\frac{n}{n+2}}}
\, d\haus
\right]^{\frac{n+2}{2n}}\,
\!\! \cdot  
\left[\int_{\partial K} \curv\, d\haus\right]^{\frac{n-2}{2n}}
\\ & =
\left[ \Theta_2 (K \right ) ]^{\frac{n+2}{2n}}\,
\cdot
\left[ n \omega _n \right  ]^{\frac{n-2}{2n}}
\,.
\end{split}
\]
Hence, for every $n \geq2$,  we have that
\begin{equation}
\label{f:es1}
\int_{\partial K} \left(\frac{\curv}{x\cdot \nu}\right)^{1/2}
\, d\haus
\leq
\Theta_2(K)^{\frac{n+2}{2n}}\, (n\, \omega_n)^{\frac{n-2}{2n}}
\leq
n\, \omega_n^{\frac{n+2}{2n}}\, |K|^{\frac{n-2}{2n}}\,,
\end{equation}
where in the second inequality we have used the $2$-affine isoperimetric inequality \eqref{f:paffine}.

From \eqref{f:c2} and \eqref{f:es1}, we get 
\begin{equation}
\label{f:below}
c \geq
\frac{|K|^2}{n^2\, \omega_n^{\frac{n+2}{n}}\, |K|^{\frac{n-2}{n}}}
= 
\frac{1}{n^2}\, \left(\frac{|K|}{\omega_n}\right)^{\frac{n+2}{n}}\,.
\end{equation}

\bigskip 
{\it -- Lower  bound for $c$ (second method).} 
We integrate  
over $\partial K$ the
pointwise overdetermined condition, after rewriting it as  
$$
|\nabla u|^2  = c\, \frac{ \curv }{x\cdot \nu }  \qquad \text{ on } \partial K\,.
$$
Then, using also the expression of $\tor (K)$ in \eqref{f:c1}, 
the isoperimetric inequality \eqref{f:BFL1} reads:
$$\frac{  c ^   { 1-\frac{1}{n+2}} \left ( \frac{n \omega _n}{n+2}  \right )  ^ { 1-\frac{1}{n+2}}   }{c \int_{\partial K} \frac{ \curv}{x \cdot \nu}} \leq \frac{ \tor  (B)^ { 1-\frac{1}{n+2}  }}{ \int_{\partial B} |\nabla u _B| ^ 2 }\,.$$ 
Rising the above inequality to power $(n+2)$, and setting for brevity 
\begin{equation*}
\Lambda (B):= \frac{ \tor  (B)^ {1-\frac{1}{n+2}  }}{ \int_{\partial B} |\nabla u _B| ^ 2 }\, ,
\end{equation*} 
we get 
\begin{equation}\label{f:uc1bis} c \geq \frac{  \left ( \frac {n \omega_n}{n +2}  \right )  ^ { n+1}  }{  \Lambda (B) ^ {n+2} \left ( \int_{\partial K} \frac{ \curv}{x \cdot \nu}  \right ) ^{ n+2} }  \,.
\end{equation} 

We now look at the integral appearing  the denominator in the right-hand side of \eqref{f:uc1bis}.
We transform it into an integral on $\mathbb S ^ { n-1}$ by the classical change of variables formula already quoted above. 
Then, by using H\"older's inequality with conjugate exponents $n$ and $\frac{n}{n-1}$, the well-known representation formula for the volume of the dual body
$|K ^o| = \frac{1}{n} \int_{\mathbb S ^ {n-1} } {h _K ^ {- n} }$, 
and the Blaschke-Santal\'o inequality \eqref{f:BS}, we obtain 
\begin{equation}\label{f:uc21} \begin{array}{ll} 
\displaystyle \int _{\partial K} \frac{ \curv}{x \cdot \nu}    
 = \int _{\mathbb S ^ {n-1} } \frac{ 1}{h _K }  
& \displaystyle  \leq \left (  \int _{\mathbb S ^ {n-1} } \frac{ 1}{h_K ^n}    \right ) ^  {\frac{1}{n}} (n \omega _ n) ^ {\frac{n-1} {n}} 
\\  \noalign{\medskip} 
& \displaystyle = \left (  n |K ^o|     \right ) ^  {\frac{1}{n}} (n \omega _ n) ^ {\frac{n-1} {n}} 
\\  \noalign{\medskip} 
& \displaystyle \leq   {   n  \omega _ n ^ {\frac{n+1} {n}} }{ { |K| ^ {-\frac{1}{n} } } }\,. 
\end{array}
\end{equation} 
We now combine \eqref{f:uc1bis} and \eqref{f:uc21}. Taking also into account that 
$\tor (B_1) = \frac{\omega _n}{n (n+2)}$, and $\int_{\partial B_1} |\nabla u _{B_1}| ^ 2  = \frac{ \omega _n }{n}$, so that 
$$\Lambda (B) ^ {n+2} = \frac{n}{ (n+2) ^ {n+1} \omega _n } \,,$$
we arrive exactly at the inequality \eqref{f:below}.  \qed

\bigskip {\it -- Conclusion.} 
By comparing the upper bound \eqref{f:above} and the lower bound \eqref{f:below}, 
we conclude that the equality sign holds in both inequalities.
In particular, the equality sign in \eqref{f:above} implies that
$K$ is a ball.
\qed 

\bigskip

\begin{theorem}\label{t:reg2} 
Let $K \in \mathcal K ^n$ have its centroid at the origin, and boundary of  class $\mathcal C^2$.
Assume that, for some constant $c>0$,  
 there exists a solution to the  following overdetermined boundary value problem on $\Omega := {\rm int }\,  K$: 

\begin{equation*}
\begin{cases}
-\Delta u = \lambda_1(\Omega) u & \text{in}\ \Omega,\\
u = 0 & \text{on}\ \partial\Omega,\\
|\nabla u|^2 {x\cdot \nu} = c\,  {\curv}
& \text{on}\ \partial\Omega\,.
\end{cases}
\end{equation*}
Then $K$ is a ball.
\end{theorem}

The proof of Theorem \ref{t:reg2} is similar to the one of Theorem \ref{t:reg1}; it  is based on the following inequalities:
\begin{itemize}

\item[(i)] The Faber-Krahn inequality
\begin{equation}
\label{f:FK}
 \lambda _ 1 (K) |K| ^ { 2/n} \geq  \lambda _ 1 (B) |B| ^ { 2/n} \,,
\end{equation}
with equality if and only if $K$ is a ball, see 
e.g. \cite[Section 3.2]{H06}.

\medskip

\item[(ii)] The isoperimetric-type inequality 
\begin{equation}\label{f:BFL}
\frac{ \lambda _ 1 (K) ^ { 3/2}  } {\int_{\partial K} |\nabla u_K| ^ 2 }
\leq 
\frac{ \lambda _ 1 (B) ^ { 3/2}  } {\int_{\partial B} |\nabla u _B| ^ 2 }
 \end{equation} 
 see \cite[Theorem 3.15]{BFL12}. 
 
 \medskip 
 \item[(iii)] The Blaschke-Santal\'o inequality \eqref{f:BS}. 

\end{itemize}

\bigskip 

{\it Proof of Theorem \ref{t:reg2}}.
Similarly as in the proof of Theorem \ref{t:reg1}, we are going to provide a matching upper and lower bound for the constant $c$ appearing in the overdetermined boundary condition.
A lower bound on $c$ is obtained by arguing as done to obtain an upper bound in the proof of Theorem \ref{t:reg1}. 
 First we integrate over $\partial K$ the overdetermined condition written as 
$$|\nabla u|^2\, x\cdot \nu = c\, \curv \qquad \text{ on } \partial K\,.
$$ Using  the  identity 
 \eqref{f:pohoL}, and a change of variables as in 
 \eqref{f:intk}, we deduce that 
\begin{equation}
\label{f:cc1}
c = \frac{2}{n\, \omega_n}\, \lambda_1(K).
\end{equation}
Then, by using the Faber-Krahn inequality \eqref{f:FK}, we obtain the lower bound 
\begin{equation}
\label{f:lc}
c \geq \frac{2 }{n} \omega _n ^ {\frac{2-n } {n} }  \lambda _ 1 (B)  |K| ^{-\frac{2}{n}}   \,,
\end{equation}
with equality if and only if $K$ is a ball.

An upper bound on $c$ is obtained by arguing as done to obtain a lower bound in the proof of Theorem \ref{t:reg1}, second method. Namely, we integrate  
on $\partial K$ the
overdetermined condition, after rewriting it as  
$$
|\nabla u|^2  = c\, \frac{ \curv }{x\cdot \nu }  \qquad \text{ on } \partial K\,.
$$
By using also the expression of $\lambda _1(K)$ in \eqref{f:cc1},
the isoperimetric inequality \eqref{f:BFL} reads:
$$\frac{  c ^   { \frac{3}{2}} \left ( \frac{n \omega _n}{2}  \right )  ^ { \frac{3}{2}}   }{c \int_{\partial K} \frac{ \curv}{x \cdot \nu}} \leq \frac{ \lambda _ 1  (B)^ {\frac{3}{2}  }}{ \int_{\partial B} |\nabla u _B| ^ 2 }\,.$$ 
Rising the above inequality to power $2$, and setting for brevity 
\begin{equation*}
\Lambda (B):= \frac{ \lambda _ 1  (B)^ {\frac{3}{2}  }}{ \int_{\partial B} |\nabla u _B| ^ 2 }\, ,
\end{equation*} 
we obtain

\begin{equation*}
c \leq \left ( \frac {2}{n \omega _n}  \right )  ^ { 3}  \Lambda (B) ^ 2 \left ( \int_{\partial K} \frac{ \curv}{x \cdot \nu}  \right ) ^ 2 \,.
\end{equation*}

For the integral appearing at the right hand side of the above inequality, 
the estimate \eqref{f:uc21} 
found in the proof of Theorem \ref{t:reg1} holds.
Hence we get
\begin{equation} \label{f:uc} 
c\leq  \frac{8  }{n} \omega _n ^ { \frac{2-n}{n} } \Lambda (B) ^ 2  { |K | ^ {-\frac{2}{n}}} \,.
\end{equation} 
To conclude the proof, it remains to show that  the expressions at the right hand sides of the lower bound  \eqref{f:lc} and of the the upper bound \eqref{f:uc} coincide. Indeed in this case  the Faber-Krahn inequality must hold as an equality, yielding that $K$ is a ball. The matching of our upper and lower bounds corresponds to the equality  
\begin{equation}\label{f:mir}2 \lambda _ 1 (B) = \int _{\partial B} |\nabla u _ B| ^ 2 \,.
\end{equation} 

The validity of \eqref{f:mir}  is checked through some direct computations  involving Bessel functions, that we enclose for the sake of completeness. We have (see for instance \cite[Section 4]{KKK})
\begin{equation}\label{f:mir1} \lambda _ 1 ( B_1) = (j _{\frac{n}{2}-1, 1} ) ^ 2\,,
\end{equation} 
where  $j _{\frac{n}{2}-1, 1}$ is the first zero of the Bessel function
$J _{\frac{n}{2}-1}$, and 
$$u _ {B_1} (r) = C  r ^ {1- \frac{n}{2} } J _{\frac{n}{2}-1} ( j _{\frac{n}{2}-1, 1}  r)\,.$$
The value of the constant $C$ is determined by imposing that $u _ { B_1}$ has unit $L ^2$-norm, yielding
$$C ^ 2 = \Big ( n \omega _n  \int_0 ^ 1 J ^2 ( j _1  r)  \, r \, dr  \Big ) ^ { -1} $$
(here and in the sequel we have written for brevity  
$J:= J _{\frac{n}{2}-1}$ and $j _ 1:=  j _{\frac{n}{2}-1 , 1}$).
Since, by known properties of Bessel functions (see \cite[11.4.5]{AS64}) we have 
 $$\int_0 ^ 1 J ^ 2 (j _ 1 r) r \, dr = \frac{1}{2} (J' (j_1) ) ^ 2 \, ,$$   
we infer that
$$C ^2 =  \frac{2}{n \omega _n}   \frac{1}{(J' (j _1 ) ) ^ 2  } \,.$$ 
Hence,
\begin{equation}\label{f:mir2}   \int _{\partial B_1} |\nabla u _ {B_1}| ^ 2 = C^ 2 j _1 ^ 2(J' (j _1 ) ) ^ 2 =  2 j _ 1 ^ 2 \,.
\end{equation} 
From \eqref{f:mir1} and \eqref{f:mir2}, we see that \eqref{f:mir} is satisfied and our proof is achived. 
 \qed

\bigskip

\begin{theorem}\label{t:reg3} 
Let $K \in \mathcal K ^n$ have its centroid at the origin, and boundary of  class $\mathcal C^2$.
Assume that, for some constant $c>0$,  
 there exists a solution to the  following overdetermined boundary value problem on the complement of $K$: 
 \begin{equation*}
\left\{
\begin{array}{ll}
\Delta u= 0 \ &\hbox{ in } \R^n \setminus  K \\
\noalign{\smallskip} 
u=1 \ & \hbox{ on  } \partial K \\
\noalign{\smallskip} 
|\nabla u|^2 {x\cdot \nu} = c  \, {\curv} \ & \hbox{ on  } \partial K
\\
\noalign{\smallskip} 
\lim \limits_{|x| \to + \infty}  u (x) = 0\,. \\
\end{array}\right.
\end{equation*}
 Then $K$ is a ball.
 \end{theorem}

Also the proof of Theorem \ref{t:reg3} follows the same strategy of Theorem \ref{t:reg1}.  
It is based on the following inequalities:
\begin{itemize} 
\item[(i)] The isoperimetric inequality for capacity \cite{posz}
\[
\frac{\Cap(K)}{|K|^{\frac{n-2}{n}}}
\geq
\frac{\Cap(B)}{|B|^{\frac{n-2}{n}}}\,,
\]
with equality if and only if $K$ is a ball.
More explicitly, since $\Cap (B_1) = n ( n-2) \omega _n$, 
\begin{equation}
\label{f:isocap}
\Cap(K) \geq  (n-2)^2 \omega_n^{\frac{2-n}{n}} |K|^\frac{n-2}{n}  \,,
\end{equation}
with equality if and only if $K$ is a ball.

\item[(ii)] The isoperimetric inequality \eqref{f:paffine} for the $p$-affine surface area in the case $p=2$. \end{itemize}

\bigskip
{\it Proof of Theorem \ref{t:reg3}}.
A lower bound for $c$ is obtained by arguing as done to obtain an upper bound in the proof of Theorem \ref{t:reg1}. Specifically, 
by   \eqref{f:pohoC} and  \eqref{f:intk}, we obtain 
 \begin{equation}
\label{f:c1cap}
c = \frac{n-2}{n\, \omega_n}\, \Cap(K).
\end{equation}
Then, by using the
isoperimetric inequality \eqref{f:isocap} for capacity, we get
\begin{equation}
\label{f:belowcap}
c \geq  (n-2)^2 \omega _n ^ { \frac{2-n}{n} }  |K| ^ { \frac{n-2}{n} }    \,,
\end{equation}
with equality if and only if $K$ is a ball.

An upper bound for $c$ is obtained by arguing as done to obtain a lower bound in the proof of Theorem \ref{t:reg1}, first method. 
We integrate  
over $\partial K$ the
overdetermined condition, after rewriting it as  
\[
|\nabla u| = c^{1/2} \, \left(\frac{\curv}{x\cdot \nu}\right)^{1/2}\,.
\]
Using the equalities $\Cap (K)= \int_{\partial K} |\nabla u|$  (see \cite[p.27]{GT}) and \eqref{f:c1cap}, we obtain 
\[
\Cap (K)  =  \frac{n\, \omega_n}{n-2}  c = c^{1/2} \int_{\partial K} \left(\frac{\curv}{x\cdot \nu}\right)^{1/2}
\, d\haus\,,
\]
i.e.
\begin{equation}
\label{f:c2cap}
c =  \Big ( \frac{n-2}{n\omega _n}   \Big  ) ^ 2 {\left[ \int_{\partial K} \left(\frac{\curv}{x\cdot \nu}\right)^{1/2}\, d\haus\right]^2}
 \,.
\end{equation}

For the integral appearing at the right hand side of the above inequality, 
the estimate \eqref{f:es1} 
found in the proof of Theorem \ref{t:reg1} (through the use of H\"older inequality and the  
$2$-affine isoperimetric inequality)  holds. Thus we get 

\begin{equation}\label{f:abovecap} 
c \leq  \Big ( \frac{n-2}{n\omega _n}   \Big  )  ^2  \Big [  n\, \omega_n^{\frac{n+2}{2n}}\, |K|^{\frac{n-2}{2n}} \Big ] ^ 2
= (n-2)^2 \omega _n ^ { \frac{2-n}{n} }  |K| ^ { \frac{n-2}{n} }  
\,.\end{equation}

Comparing \eqref{f:belowcap}  and \eqref{f:abovecap}, 
we  see that our lower and upper bounds match each other, implying in particular that  \eqref{f:belowcap} must hold with equality sign, and hence that 
$K$ is a ball.
\qed

\section{Rigidity results for arbitrary convex bodies}\label{sec:cone}

In this section we drop any smoothness assumption and we deal with arbitrary centred convex bodies having some cone variational measure equal to that of a ball.

\begin{theorem}\label{t:nonreg1} 
Let $K \in \mathcal K ^n$ have its centroid at the origin, and let $\tau _K$ be its cone torsion measure according to Definition \ref{def:cone}. Assume that, for some positive constant $c$, 
 \begin{equation}\label{f:equtau} \tau _K   = c \, \mathcal H ^ { n-1} \res \mathbb S ^ { n-1} 
 \text{ as measures on } \mathbb S ^ {n-1}\,.
\end{equation}
Then $K$ is a ball.

\end{theorem}

The proof of Theorem \ref{t:nonreg1} requires some preliminaries about convex bodies which have absolutely continuous surface area measure (or equivalently which admit a curvature function); we gather them in Proposition  \ref{p:Hug}, 
relying on some results proved in \cite{Hugc,Hug1,Hug2}.   

We recall that,  if a convex body $K$ has absolutely continuous surface area measure, by definition there exists a unique non-negative function $f_K \in L^1(\mathbb S^{n-1},   \haus \res \mathbb S^{n-1}  )$, called  {\it curvature function} of $K$,  such that its surface area measure satisfies $S_K = f _K \haus \res \mathbb S^{n-1} $. Equivalently, we have 
\[
V _ 1(K, L) := \lim_{t \to 0^+ } \frac{V (K + t  L)  - V (K)}{t}  = \frac{1}{n} \int _{\mathbb S ^ {n-1} } h _ L (\xi) f_K (\xi) d \mathcal H ^ {n-1}  (\xi) \qquad \forall L \in \mathcal K ^n \,. 
\] 
 
Below we follow the usual convention that 
$\nu\colon\partial K \to \mathbb S^{n-1}$ is a possibly multivalued map associating
to every $x\in\partial K $ the unit vectors of the normal cone to $\partial K $
at $x$. 

Moreover, using the notations in the cited papers of Hug, we set 
\begin{itemize}
\item[--] $\mathcal{M}(K)$:= the set of points  $x\in \partial K$  such that $\partial K$ is second order differentiable at $x$
(so that $\nu(x)$ is a singleton and 
$\curv(x)$ is well-defined as the product of the principal curvatures);
\medskip

\item[--] $(\partial K)_+$:= the set of points $x\in\partial K$ such that there exists an
internal tangent ball touching $\partial K$ at $x$ (so that $\curv(x)$ is finite);

\medskip
\item[--] $\exp^* K$:= the set of  points $x\in\partial K$ such that there
exists a closed ball $B$ containing $K$ with $x\in\partial B$ (so that $\curv(x) > 0$).
\end{itemize}

\begin{proposition}\label{p:Hug}
Assume that $K \in \mathcal K ^n$ has absolutely continuous surface area measure, 
and let $f_K$ be its curvature function.  Then:
\begin{itemize} 
\item[(a)]  Denoting by $r_1(\xi ), \ldots, r_{n-1}(\xi )$ the generalized radii of curvature
of $\partial K$ at $\nu^{-1}(\xi )$, namely the eigenvalues of $D^2 h_K(\xi ) | _{\xi ^\perp}$, we have
\begin{equation}
\label{f:fr}
f_K(\xi ) = r(\xi )  :=  \prod_{j=1}^{n-1} r_j(\xi )
\qquad
\text{for $\haus$-a.e.\ $\xi \in  \mathbb S^{n-1}$}.
\end{equation}
\item[(b)] Setting 
$$
\reg[K] :=
\mathcal{M}(K) \cap (\partial K)_+ \cap \exp^* K\,, 
$$
 we have that
$ \haus(\partial K \setminus\reg[K]) = 0$,
 and $\nu$ is a bijection from $\reg[K]$ to $\nu(\reg[K])$.

\end{itemize}
\smallskip 

If, in addition, $f_K(\xi) > 0$ for $\haus$-a.e.\ $\xi \in \mathbb S^{n-1}$,
then:

\smallskip 
\begin{itemize}

\item[(c)] $\nu(\reg[K])$ has full measure in $\mathbb{S}^{n-1}$,
i.e.\ $\haus(\mathbb S^{n-1} \setminus \nu(\reg[K])) = 0$.

\smallskip 
\item[(d)] After possibly removing a $\haus$-null set from $\reg[K]$,  for every $\xi \in \nu(\reg[K])$, $h_K$ is second order differentiable
at $\xi $, $x := \nabla h_K(\xi ) \in \reg[K]$ and $\curv(x) \, r(\xi ) = 1$.

\smallskip
\item[(e)]
For every nonnegative function $\psi \in L^1(\partial K,  \haus \res \partial K  )$ it holds that
\[
\int_{\partial K} \psi \, d\haus
=
\int_{\mathbb S^{n-1}} \frac{\psi (\nu^{-1}(\xi ))}{\curv(\nu^{-1}(\xi ))}\, d\haus(\xi )
\,.
\]
\end{itemize}
\end{proposition}

\begin{proof} 
The equality \eqref{f:fr} is proved in \cite[formula (2.8)]{Hug1}.

Statement (b) is consequence of the 
facts that $\haus(\partial K \setminus \mathcal{M}) = 0$,
$\haus(\partial K \setminus (\partial K)_+) = 0$ 
and 
$\haus(\partial K \setminus \exp^* K) = 0$,
following respectively from 
Alexandroff theorem \cite{Alex}, \cite{McMullen},
and \cite[Theorem~3.7(c)]{Hug2}.

Assume now that $f_K > 0$ $\haus$-a.e.\ on $\mathbb S^{n-1}$. 
Recall that the
curvature measure of $K$ is the measure $C_K$, supported on $\partial K$, defined by
\[
C_K (E) := \haus(\nu(E))
\qquad
\text{for every Borel set $E\subseteq \partial K$}.
\] 
By \eqref{f:fr} and Theorem~2.3 in \cite{Hug2},
it follows that  $C_K$ is absolutely continuous with respect to  $\haus\res\partial K$.
Since, by  statement (b), the set $\partial K \setminus \reg[K]$ is  $\haus$-negligible,
we conclude that also statement (c) holds, namely $\mathbb S^{n-1} \setminus \nu(\reg[K])$ is
$\haus$-negligible.
To prove statement  (d), we consider the subset of $\mathbb S ^ {n-1} $ given by   
\[
\Sigma := \big \{
\xi \in \nu(\reg[K])\colon \text{$h_K$ is second order differentiable at $\xi $}
\big \}\,,
\]
and we replace
$\reg[K]$ with the possibly smaller set $\nu^{-1}(\Sigma)$. 
Then statement (d) follows from (c) and Alexandroff theorem.

Finally, let us prove (e). Following \cite{Hugc}, for every $r>0$ let $(\partial K)_r$
be the set of all points of $x\in\partial K$ such that there exists an internal
tangent ball of radius $r$ touching $\partial K$ at $x$.
From Lemma~2.3 in \cite{Hugc}, the map $\nu\res (\partial K)_r$ is
Lipschitz continuous, and its approximate $(n-1)$-dimensional Jacobian is
\[
\text{ap}\, J_{n-1} \nu\res (\partial K)_r (x)
= \curv(x),
\qquad
\text{for $\haus$-a.e.\ $x\in (\partial K)_r$.}
\]
Recalling that $\curv(x) > 0$ for $\haus$-a.e.\ $x\in\partial K$,
for every positive $n\in \N$
we can apply Federer's coarea formula (see \cite[Theorem~3.2.22]{Fed})
to the non-negative function 
\[
h_n(x) := \frac{\psi (x)}{\curv(x)}\, \chi_{(\partial K)_{1/n}}(x),
\] 
obtaining
\[
\int_{\partial K} h _n\, \curv\, d\haus
=
\int_{(\partial K)_{1/n}} h_n\, \curv\, d\haus
= 
\int_{\nu\left((\partial K)_{1/n}\right)} \frac{ \psi \circ\nu^{-1}}{\curv\circ\nu^{-1}}\, d\haus\,.
\]
Since $(\partial K)_+ = \bigcup_{n} (\partial K)_{1/n}$,
the change of variable formula (e) follows by using (c), and applying Lebesgue monotone
convergence theorem.
\end{proof}

\bigskip

We are now in a position to give the 

\bigskip

{\it Proof of Theorem \ref{t:nonreg1}.}  
The idea is to follow the same proof line of Theorem \ref{t:reg1}. However, this cannot be done directly,  
since we do not have any longer a pointwise identity holding along the boundary. So,  
we have to prove first of all that the constant $c$ appearing in \eqref{f:equtau} can be still expressed by the two different formulas \eqref{f:c1} and \eqref{f:c2}. 
 To obtain  formula \eqref{f:c1}, it is enough to observe that  the two measures in the overdetermined conditions must have the same total variation: using \eqref{f:poho}, we obtain
\[
|\tau _K| =  (n+2) \tor (K),
\]  
and hence
$$c =  \frac{ |\tau _K|  }{\mathcal H ^ {n-1} ( \mathbb S ^ {n-1}) } =  \frac{n+2}{n\, \omega_n}\, \tor(K) \,.$$ 

To obtain also formula \eqref{f:c2}, we need first to prove the following
 
 \smallskip
Claim:  {\it the equality \eqref{f:equtau} implies that $K$ has absolutely continuous surface area measure, and in addition its curvature function is strictly positive.} 
\smallskip 

This amount to show that, if the equality \eqref{f:equtau} holds,  for a Borel set $E \subset \mathbb S ^ {n-1}$, the following implications hold:
\begin{equation}\label{f:implication1} 
\mathcal H ^ {n-1} \res \mathbb S ^ {n-1}  ( E ) = 0 \quad \Rightarrow \ \surf (E) = 0\,;
\end{equation}  
\begin{equation}\label{f:implication2} 
\mathcal H ^ {n-1} \res \mathbb S ^ {n-1}  ( E ) > 0 \quad \Rightarrow \ \surf (E) > 0\,.
\end{equation}  
Indeed, \eqref{f:implication1} implies that  $K$ admits a curvature function $f_K$; then,  since
 $\surf = f_K\, \haus\res \mathbb S^{n-1}$, \eqref{f:implication2} implies that $f_K$ is strictly positive $\haus$-a.e.\ on $\mathbb S^{n-1}$.

Let $E  \subset \mathbb S ^ {n-1}$ be a Borel set with $\mathcal H ^ {n-1} \res \mathbb S ^ {n-1}  ( E ) = 0$. 
By \eqref{f:equtau}, we have $\tau _K(E) = 0$. 
Recalling that, by definition,  we have 
\begin{equation}
\label{f:tau}
\tau _K (E) = \int _{\nu ^ { -1} ( E)} |\nabla u| ^ 2  x \cdot \nu \, d \mathcal H ^ {n-1}\,,
\end{equation}
 and
 \[
\surf ( E) = \mathcal H ^ {n-1} (\partial K \cap \nu ^ { -1} (E) ) \,,
\] 
the  implication \eqref{f:implication1} follows from the fact that 
$$ |\nabla u| ^ 2  x \cdot \nu  > 0 \qquad \mathcal H ^ {n-1}\text{-a.e. on } \partial K \,.$$ 
Indeed, the term $x \cdot \nu$ remains strictly positive since by assumption $K$ contains the origin in its interior, while the term $|\nabla u| ^2$ remains strictly positive 
by Hopf boundary point lemma (since  $K$ admits an inner touching ball at  $\mathcal H ^ {n-1}$-a.e. $x \in \partial K$,
see e.g.\ \cite{McMullen}).

Let $E  \subset \mathbb S ^ {n-1}$ be a Borel set with $\mathcal H ^ {n-1} \res \mathbb S ^ {n-1}  ( E ) > 0$. 
From \eqref{f:equtau}, 
it follows that $\tau_K(E) = c\, \haus(E) > 0$. In turn, by \eqref{f:tau},
this implies that $\haus(\nu^{-1}(E)) > 0$, or equivalently that $S_K ( E) >0$, proving \eqref{f:implication2}.

\medskip
Since we have just proved that $K$ admits a positive curvature function $f_K$, recalling also the definition of cone torsion measure $\tau _K$, 
the equality \eqref{f:equtau} can be reformulated as   
\[
h _K \nu _{\sharp } (|\nabla u | ^ 2 \mathcal H ^ { n-1} \res \partial K) = \frac{c}{f _K } \nu _\sharp ( \mathcal H ^ { n-1} \res \partial K ) \,.
\] 
 
Namely, for every Borel set  $E \subset \mathbb S ^ {n-1}$, it holds that
\begin{equation}
\label{f:meas0}
\int _{\nu ^ { -1} ( E)} x \cdot \nu |\nabla u | ^ 2  \, d \mathcal H ^ { n-1} =   c\, 
\int _{\nu ^ { -1} ( E)} \frac{1}{f _ K \circ \nu} \, d\haus\,.
\end{equation}
We now invoke Proposition \ref{p:Hug}. By statements (a) and (d), we can rewrite the right hand side of the above equality as
$$\int _{\nu ^ { -1} ( E)} \frac{1}{f _ K \circ \nu} \, d\haus
=  \int _{\nu ^ { -1} ( E)} \curv\, d\haus\,,
$$
so that \eqref{f:meas0} turns into 
\begin{equation}
\label{f:meas00}
\int _{\nu ^ { -1} ( E)} x \cdot \nu |\nabla u | ^ 2  \, d \mathcal H ^ { n-1} =   c\, 
\int _{\nu ^ { -1} ( E)} \curv\, d\haus\,.
\end{equation}

By Proposition \ref{p:Hug}, statements (b) and (c), we know that 
$\nu\res\reg$ is a bijection from $\reg$ to $\nu(\reg)$,
with $\haus(\partial K \setminus\reg) = 0$ and
$\haus( \mathbb S^{n-1}\setminus \nu(\reg)) = 0$. 
In view of this fact, the integral equality \eqref{f:meas00} yields the pointwise equality 
\[
|\nabla u| = c^{1/2} \, \left( \frac {\curv(x)}{x\cdot \nu(x)}\right)^{1/2} \qquad  \mathcal H ^ {n-1}\text{-a.e. on } \partial K \,.
\]
Integrating over $\partial K$, we see that \eqref{f:c2} holds.

Having established the validity of \eqref{f:c1} and \eqref{f:c2}, the remaining of the proof can proceed in the analogous way as in Theorem \ref{t:reg1}. 
 Indeed, the inequalities used  in the proof of Theorem~\ref{t:reg1} do not need any smoothness assumption. 
In particular, for the 
 $2$-affine isoperimetric inequality  for convex bodies which admit a curvature function, we refer to \cite[Theorem~3]{ludwig} (see also \cite[Remark~p.~296]{Hugc}). 
\qed

\bigskip

As well as the inequalities used in the proof of Theorem \ref{t:reg1},  also the inequalities used in the proof of Theorems \ref{t:reg2} and \ref{t:reg3} do not need any smoothness assumption. Thus we obtain the following results for convex bodies with the same cone eigenvalue measure or the same cone capacity measure as a ball. We omit their proofs since they are analogous to the one of Theorem \ref{t:nonreg1} detailed above. 

\begin{theorem}
Let $K \in \mathcal K ^n$ have its centroid at the origin, and let $\sigma _K$ be its cone eigenvalue measure according to Definition \ref{def:cone}. Assume that, for some positive constant $c$, 
\begin{equation*}
\sigma _K   = c \, \mathcal H ^ { n-1} \res \mathbb S ^ { n-1} 
 \text{ as measures on } \mathbb S ^ {n-1}\,.
\end{equation*}
Then $K$ is a ball.
\end{theorem}

\begin{theorem}
Let $K \in \mathcal K ^n$, $n\geq 3$, have its centroid at the origin, and let $\eta _K$ be its cone capacity measure according to Definition \ref{def:cone}. Assume that, for some positive constant $c$, 
\begin{equation*}
\eta _K   = c \, \mathcal H ^ { n-1} \res \mathbb S ^ { n-1} 
 \text{ as measures on } \mathbb S ^ {n-1}\,.
\end{equation*}
Then $K$ is a ball.
\end{theorem}

 \section{Ultimate shapes of variational flows}\label{sec:flow}  
In this section we  show that, in the smooth centrally symmetric setting, the convex bodies analysed in Section~\ref{sec:overdet}
can be identified with the ultimate shapes of new variational flows, under the basic assumption that they admit a solution. We present a unified treatment, by considering the evolution problem  
\begin{equation}
\label{f:evolutionF}
\begin{cases}
 \displaystyle  \frac{ \partial h}{\partial t} ( t, \xi) =   - \costa F (C( t))  \frac{ \curv (t, \nu_t  ^ { -1} ( \xi ) )}{|\nabla u( t, \nu_t ^ { -1}  (\xi)) | ^ 2}  & \text{ on } (0, T) \times {\mathbb S} ^ { n-1} ,\\
h ( 0, \xi) = h _ 0 ( \xi)  & \text{on} \ {\mathbb S} ^ {n-1},\\
h ( t, \xi) \leq \costb& \text{on} \ (0, T) \times {\mathbb S} ^ {n-1} \,.
\end{cases}
\end{equation}
where $F$ may denote either the  torsional rigidity, or the principal 
Dirichlet Laplacian eigenvalue, or the Newtonian capacity (the latter for $n \geq 3$), 
and $u$ is the corresponding ground state. 

Let us remark that,  in terms of the parametrization of $\partial C (t)$  by its inverse Gauss map 
$X ( t, \cdot) :  \mathbb S ^ {n-1}  \ni \xi \to  \nu _t^ { -1} ( \xi) \in \partial C ( t)$, 
the evolution equation in \eqref{f:evolutionF} can be written  as
$$ \frac{ \partial X}{\partial t} (t, \xi) =  - \frac{\curv (t, X)}{|\nabla u| (t,X) ^ 2}  F (C(t)) \,  \xi  \qquad \text{ on } (0, T) \times \mathbb S ^ { n-1} $$
(indeed, we have $h ( t, \xi) = \nu _t^ { -1} ( \xi) \cdot \xi$, so that $\frac{ \partial h}{\partial t} ( t, \xi)= \frac{ \partial X}{\partial t} ( t,  \nu_t ^ { -1} (\xi)) \cdot \xi $). 

\smallskip
Thus, up to the factor $F (C(t))$ (which is just a rescaling allowing to have global existence in time), 
the unique crucial difference  bewteen \eqref{f:evolutionF} and the classical Gaussian curvature flow  is the presence of the squared modulus of the ground state gradient. 

We denote by $\alpha$ the homogeneity degree of $F$ under domain dilations   (see Section \ref{sec:cone}).
The following statement deals with the cases of torsion and first eigenvalue, see however Remark \ref{r:capacity} for the case of capacity. 

\begin{theorem}\label{t:firey}  Let $F = \tor$ or $F = \lambda _ 1$. 
Assume that, for some $T> 0 $, there exists a unique family of convex bodies $C ( t)$, defined for $t \in [ 0, T)$, 
which are smooth (at least of class $\mathcal C ^3$) and such that their support function $\xi \mapsto h ( t, \xi)$ satisfies \eqref{f:evolutionF}. 
Then, setting
$\gamma = \frac{\costa n \omega_n}{|\alpha|}$,
\begin{itemize}

\item[(i)]  it holds that $F (t) = F _0 \exp ( - \g t)$, where $F _0$ is the energy at $t= 0$;  in particular, the family $C ( t)$ is strictly decreasing by inclusion; 
\smallskip

\item[(ii)]    the flow is defined  for every $t \geq 0$ (i.e., $T = + \infty$), 
and the sets 
 $C ( t)$ are strictly convex bodies; 
\smallskip

\item[(iii)]  the following entropy is decreasing along the flow:
$$\mathcal E ( t):= \int _ { \mathbb S ^ {n-1} } \log h ( t, \xi) \, d \mathcal H ^ {n-1} ( \xi) +  \gamma n \omega _n  t \,;$$

\item[(iv)]  assuming in addition that $C (0)$ is centrally symmetric, we have that  $C (t) $ is centrally symmetric for every $t \geq 0$; 
moreover, if $\{t _n \} \to + \infty$, up to subsequences $ \exp ({\g} t_n) C ( t_n)$ converge in Hausdorff distance to a centrally symmetric convex body, called an {\rm ultimate shape} for problem \eqref{f:evolution}, whose cone torsion  measure or  cone eigenvalue measure  is absolutely continuous with respect to  $\mathcal H ^ {n-1} \res \mathbb S ^ {n-1}$, with constant density. 
\end{itemize}
\end{theorem}

\begin{remark}\label{r:capacity} As it can be seen by direct inspection of the proof below, for $F = \Cap$,  thanks to the Brunn-Minkowski inequality proved in \cite{bor2}, statement (i) continues to hold, but in principle statement (ii) may fail, because
the positivity of the capacity does not imply that the corresponding convex body is nondegenerate.  However, if it occurs that the flow is defined for all times, 
then also statements (iii)-(iv) hold true. 
\end{remark}

\proof  We follow the proof line of  \cite[Thm.~1, Thm.~2]{firey}.

 (i) Set 
$$F_1 (K, L) :=\frac{1}{\alpha}  \lim _{t \to 0^+ } \frac{F( K + t L ) - F(K)  }{t}  = \frac{1}{|\alpha|} \int _{\mathbb  S ^ { n-1}} h _L \, d \mu _K\,,$$ 
where $\mu _K$ is the first variation measure given by \eqref{f:mufir}; recall that, by the Brunn-Minkowski inequality  satisfied by $F$ (see \cite{B85} for torsion and \cite{BrLi} for the first eigenvalue), it holds that
\begin{equation}\label{f:M1} F _ 1 (K, L)  \geq F (K) ^ { 1- \frac{1}{\alpha} } F (L) ^ {\frac{1}{\alpha}}  \,.
\end{equation} 

Rewrite the equation as 
$$ \frac{|\nabla u( t, \nu_t ^ { -1}  (\xi)) | ^ 2}  { \curv (t, \nu_t ^ { -1} ( \xi ) )}  \lim _{ s \to 0} \frac{1}{s}  \big [ h ( t+ s, \xi) - h ( t , \xi) \big ]  = - \costa F (t) \,.$$ 
Integrating  over $\mathbb S ^ {n-1}$ with respect to $\xi$,  we obtain
$$|\alpha |\lim _{ s \to 0 } \frac{  F  _ 1 (  t,  t+s) - F (t )  }{s} = - \costa n \omega_n F     (t)\,, $$
where we used for shortness the notation $ F _1 ( t, t') :=  F  _ 1 ( C(t), C(t'))$,  $ F  (t):=  F  (C(t))$. 
  
Then, using \eqref{f:M1}, we arrive at
\begin{equation}\label{f:ttorder}  - \frac{  \costa n \omega_n }{|\alpha|}  F  (t)  \geq  F  ( t) ^ { 1 - \frac{1}{\alpha} } \frac{d}{dt} \Big (  F  ^ {\frac{1}{\alpha } }( t) \Big ) =  \frac{1}{\alpha} \frac{d}{dt}  F  ( t)  \,.\end{equation} 

Consider now the convex body $\tc ( t)$  with support function $\th ( t, \xi):= h ( t, \xi) \exp ({\g} t)$. 
Denoting by $\tu$, $\tn_t$, $\tk$, $\ttor$ respectively its torsion function, Gauss map, Gaussian curvature, and torsional rigidity, it holds that
$$
\begin{array}{ll} & \displaystyle \tk ( t , \tn _t ^ { -1}  (\xi) ) = \exp  \big( {- (n-1) \gamma t }\big )  \curv ( t , \nu _t ^ { -1}  (\xi) ) \,, \\
\noalign{\medskip} 
& \displaystyle |\nabla \tu ( t , \tn_t  ^ { -1}  (\xi) ) | ^ 2 = \exp  \big( {  (\alpha-n)\gamma   t } \big )  |\nabla u ( t , \nu_t  ^ { -1}  (\xi) )  | ^ 2\,, \\
\noalign{\medskip} 
& \displaystyle \ttor (t) = \exp  ( \alpha \gamma   t )   F  (t)\, .
\end{array} 
$$
Taking
$\gamma = \frac{\costa n \omega_n}{|\alpha|}$,
by \eqref{f:ttorder}, we have 
\begin{equation}\label{f:segno1} ({\rm sign}\,  \alpha) \frac{d}{dt} \ttor (t) \leq 0 \,.
\end{equation} 
The equation satisfied by $\th$ is 
\begin{equation}\label{f:eqtilde} \frac{|\nabla \tu( t, \tn_t  ^ { -1}  (\xi)) | ^ 2}  { \tk (t, \tn _t ^ { -1} ( \xi ) )}   \Big [ \frac{\partial \th } {\partial t}( t, \xi)  -  \g \th (t, \xi) \Big ]  = - \costa\ttor  (t)  \,.
\end{equation} 
Integrating this equation  on $\mathbb S^{n-1}$, by the choice of $\gamma$ and the representation formula for $\ttor (t)$,  we infer that 
$$\int_{ \mathbb S ^ {n-1} }   \frac{|\nabla \tu( t, \tn_t  ^ { -1}  (\xi)) | ^ 2}  { \tk (t, \tn_t  ^ { -1} ( \xi ) )}    \frac{\partial \th } {\partial t}( t, \xi)   \, d \mathcal H ^ {n-1} (\xi) = 0\,.$$
Consequently, differentiating under the sign of integral (thanks to the smoothness assumption made on $C ( t)$ and standard elliptic boundary regularity), 
$$
\frac{d}{dt} \ttor ( t) = \frac{1}{|\alpha|} \int_{\mathbb S ^ {n-1} }  \frac{\partial  } {\partial t}   \Big [ \frac{|\nabla \tu( t, \tn _t ^ { -1}  (\xi)) | ^ 2}  { \tk (t, \tn_t  ^ { -1} ( \xi ) )}  \Big ]   \th ( t, \xi)  \, d \mathcal H ^ {n-1} (\xi) \,,$$
that we may rewrite as 
$$
\frac{d}{dt} \ttor ( t) =   \lim _{s \to 0 } \frac{ \ttor_1 ( t+s,t ) - \ttor (t)  }{s}  \,, $$
 where we have set $\ttor_1 (t+s, t ) :=  F  _1 ( \tc ( t+s),  \tc ( t))$.  Then, using \eqref{f:M1} in a similar way as above, we obtain 
$$ \frac{d}{dt} \ttor ( t) \geq  \Big [ \ttor ( t) ^ { \frac{1}{\alpha} } \frac{d}{dt} \Big ( \ttor ^ {1-\frac{1}{\alpha} } ( t) \Big ) \Big ]  = \Big ( 1- \frac{1}{\alpha} \Big ) \frac{d}{dt} \ttor  ( t)  \,, $$
which shows that  
\begin{equation}\label{f:segno2} 
({\rm sign}\,  \alpha) \frac{d}{dt} \ttor (t) \geq 0 \,.
\end{equation}

By combining \eqref{f:segno1} and \eqref{f:segno2}, we see that $\ttor (t)$ is equal to a constant, precisely to  $ F _0 := \ttor ( 0) =  F  (0) $, and hence 
$$ F  ( t) =  F _0 \exp ( - \alpha \g t)\,.$$ 

Since $F ( t) >0$ for every $t \in [0, T)$, from the equation we see in particular that the family of convex sets $C (t) $ is strictly decreasing with respect to $t$.

\smallskip 

(ii) We recall that, for a convex body $K$, we have the inequalities
\begin{equation}\label{f:stimagiu}
\tor(K) \leq \frac{|K|^\frac{n+2}{n}}{n(n+2) \omega_n^{\frac{2}{n}}} \qquad \text{ and } \qquad \lambda _ 1 ( K) \geq \frac{ \pi ^ 2 }{  
(\text{minimal width} (K)  ) ^ 2}
\end{equation}
(the first one is the Saint-Venant inequality \eqref{f:isotor},  while for the second one we refer e.g. to \cite[Proposition 11]{BF19}). 
It follows that 
$C(t)$ is a nondegenerate convex body for every $t$, and hence that  $T = + \infty$.  
The strict convexity of $C ( t) $ follows from the strict positivity of its Gaussian curvature resulting from the equation.

\smallskip
(iii)  Dividing equation \eqref{f:eqtilde} by $ \frac{|\nabla \tu( t, \tn_t ^ { -1}  (\xi)) | ^ 2}  { \tk (t, \tn_t ^ { -1} ( \xi ) )}  \th ( t ,\xi)$ and recalling that $\ttor (t)\equiv F _0$, we get 
$$\frac{d}{dt} \log \th (t, \xi) - {\g} = - \costa F _0  \frac  { \tk (t, \tn_t  ^ { -1} ( \xi ) )} {|\nabla \tu( t, \tn_t ^ { -1}  (\xi)) | ^ 2} \frac{1}{  \th ( t , \xi)} \,.$$ 
By applying the H\"older inequality $\big( \int _{\mathbb S ^ {n-1} } f \big ) \big ( \int _{\mathbb S ^ {n-1} } \frac{1}{f}  \big )   \geq n ^ 2 \omega _n ^ 2 $, with 
$f =  \frac  { \tk (t, \tn_t  ^ { -1} ( \xi ) )} {|\nabla \tu( t, \tn_t  ^ { -1}  (\xi)) | ^ 2} \frac{1}{  \th ( t , \xi)} $, we get 
$$\begin{array}{ll}  \displaystyle \int_{\mathbb S ^ {n-1}} \frac{d}{dt} \log \th (t, \xi) \, d \mathcal H ^ {n-1} (\xi)  
& \displaystyle \leq    { \gamma n \omega _n  }  - \costa F _0  n ^ 2 \omega _n ^ 2 \Big  [  \int _{ \mathbb S ^ {n-1}}  \frac {|\nabla \tu( t, \tn _t ^ { -1}  (\xi)) | ^ 2}  { \tk (t, \tn_t  ^ { -1} ( \xi ) )}   \th ( t , \xi) \Big ]^ { -1} 
\\ \noalign{\medskip} 
& \displaystyle =  \frac {\costa n^2  \omega _n ^2 }{|\alpha|}   - \costa F _0  n ^ 2 \omega _n ^ 2 \big  ( |\alpha|  F _0  \big )^ { -1}  = 0\,. 
\end{array} 
 $$ 
By interchanging derivative and integral on $\mathbb S ^ {n-1}$, we obtain that the entropy $\mathcal E ( t) $ decreases along the flow. 

\smallskip
(iv) The central symmetry of $C (t)$ corresponds to the condition $h ( t, \xi) = h ( t, - \xi)$, which is satisfied since the map $h ( t, \xi)$ is a solution to problem \eqref{f:evolutionF}, and by  assumption such problems admits a unique solution.  Now,  integrating on $[0, t]$ the inequality $\mathcal E' \leq 0$, we obtain 
$$ \mathcal E ( t)= \int_{\mathbb S ^ {n-1}} \log \th (t, \xi) \, d \mathcal H ^ {n-1} (\xi)    \leq \int_{\mathbb S ^ {n-1}} \log \th_0 ( \xi) \, d \mathcal H ^ {n-1} (\xi)  = \mathcal E ( 0)   \,.$$ 
By exploiting the central symmetry of $\tc(t)$ this implies,  by arguing as in \cite[proof of Thm.~2]{firey}, that the convex bodies $\tc(t)$ lie into some fixed ball independent of $t$;
moreover, applying the inequalities \eqref{f:stimagiu} it follows that, for every $t \geq 0$, $\tc (t)$ contains a fixed ball centered at the origin.  So there exist $R>r>0$ independent of $t$ such that 
\begin{equation}\label{f:contenimenti} 
B _ r (0) \subseteq \tc (t) \subseteq B _ R ( t) \qquad \forall t \in [0, +\infty). 
\end{equation} 
In particular this implies that the monotone decreasing map $\mathcal E (t)$ is bounded from below, and so it converges to a finite limit as $t \to + \infty$; thus we have 
\begin{equation}\label{f:derentropy} \lim _{t \to + \infty} \frac{d}{dt} \mathcal E ( t) =   \lim _{t \to + \infty} \frac{d}{dt} \int_{\mathbb S ^ {n-1}}  \log \th (t, \xi) \, d \mathcal H ^ {n-1} (\xi)   = 0\,.
\end{equation}
Set
$$g ( t, \xi):= 1 - \frac{1}{\g}\frac{  \partial }{\partial t} \log \th (t, \xi) \,. $$ 
We have 
$$\begin{array}{ll} & \displaystyle \int_{\mathbb S ^ {n-1} }g ( t, \xi) \, d \mathcal H ^ {n-1} (\xi) = n \omega_n -  \frac{1}{\g} \int_{\mathbb S ^ {n-1} } \frac{  \partial }{\partial t} \log \th (t, \xi) \, d \mathcal H ^ {n-1} (\xi) \\ \noalign{\medskip} 
& \displaystyle \int_{\mathbb S ^ {n-1} } \frac{1}{g  ( t, \xi)} \, d \mathcal H ^ {n-1} (\xi)   =   \frac{\g }{  \costa F_0} \int_{\mathbb S ^ {n-1} } \th ( t, \xi)  \frac{|\nabla \tu( t, \tn_t ^ { -1}  (\xi)) | ^ 2}  { \tk (t, \tn_t ^ { -1} ( \xi ) )}   \, d \mathcal H ^ {n-1}  = { n \omega _n }  \,,  \end{array} $$ 
 where to compute the integral of $g ^{-1}$ we have used  \eqref{f:eqtilde}.   It follows that
 $$\int_{\mathbb S ^ {n-1} } \Big ( \sqrt g - \frac{1}{\sqrt g }  \Big ) ^ 2 \, d \mathcal H ^ {n-1} (\xi) = -  \frac{1}{\g} \int_{\mathbb S ^ {n-1} } \frac{  \partial }{\partial t} \log \th (t, \xi) \, d \mathcal H ^ {n-1} (\xi)  $$ 
 and hence by \eqref{f:derentropy} we have that $g ( t, \xi) \to 1$ $\mathcal H ^ {n-1}$-a.e.\ on $\mathbb S ^ {n-1}$ as $t \to + \infty$. 
 By \eqref{f:contenimenti}, the support functions $\th ( t, \cdot)$ are bounded from above and from below on $\mathbb S ^ {n-1}$ by positive constants independent of $t$. It follows that
  \begin{equation}\label{f:vanishing} \lim _{t \to + \infty} \frac{ \partial \th }{\partial t}  ( t, \xi ) = 0 \qquad \mathcal H ^ {n-1}\text{-a.e.\ on } \mathbb S ^ {n-1}  \,. 
 \end{equation} 
Take now a sequence $\{t _n \} \to + \infty$. By \eqref{f:contenimenti}  and 
by the compactness of the Hausdorff distance,  $\tc ( t _n )$  converges to a nondegenerate convex body $\tc _\infty$. 
Then the cone variational measures of $\tc ( t _n )$ converge weakly* to the cone variational measure of $\tc _\infty$ (see Remark~\ref{r:weak}). 
Thus, if we set
$$\psi _n ( \xi):=  \frac{|\nabla \tu( t _n, \tn_{t _n}  ^ { -1}  (\xi)) | ^ 2}  { \tk (t_n, \tn _{t_n} ^ { -1} ( \xi ) )}     \th (t _n, \xi) \,$$
we have that $\psi _n$ converge weakly in 
$L ^ 1 (\mathbb S ^ {n-1},  \mathcal H ^ {n-1}\res \mathbb S ^ {n-1})$ to $\psi _\infty$, being $\psi _\infty$ the density of the cone variational measure of $\tc _\infty$. 

On the other hand, 
by passing to the limit as $n \to + \infty$ in the equation \eqref{f:eqtilde} written at $t = t _n$, recalling that $F ( t _n) = F_0$, and exploiting \eqref{f:vanishing}, we obtain the following pointwise convergence
\begin{equation}\label{f:limpuntuale}   \lim _n \psi _n (\xi)   =  \frac{ \costa }{\g} F_0  \qquad \mathcal H ^ {n-1}\text{-a.e. on } \mathbb S ^ {n-1} \,.
\end{equation} 
Moreover, we have  
$$ \lim _n  \int _{\mathbb S ^ {n-1}} \psi _n (\xi)  \, d \mathcal H ^ {n-1} = \int _{\mathbb S ^ {n-1}} \psi _\infty (\xi)  \, d \mathcal H ^ {n-1} \,,
$$
 (see respectively \cite[eq.\ (23)]{CoFi} for torsion, \cite[eq.\ (3.15)]{Je96} for capacity and \cite[Section 7]{Je}  for principal eigenvalue), 
 so that $\psi _n$ converge to $\psi _\infty$ strongly in $L ^ 1 (\mathbb S ^ {n-1},  \mathcal H ^ {n-1}\res \mathbb S ^ {n-1})$.   

In particular, from \eqref{f:limpuntuale} we conclude that $\psi _\infty = \frac{ \costa }{\g} F_0$, namely 
the cone variational measure of $\tc _\infty$ 
 is a constant multiple of $\mathcal H ^ {n-1} \res \mathbb S ^ {n-1}$. \qed

\section{Concluding remarks}\label{sec:final}

\begin{remark}
In dimension $n = 2$ and for $K$ smooth, the existence of a solution to the logarithmic Minkowski problem~\eqref{f:minpbt0} easily follows from the results proved in \cite[Section~6]{BLYZ}. 
Indeed, it is not restrictive to  assume $\tor ( K ) = 1$ and to minimize over convex bodies with  unit torsion. Let $\{L _j\}$ be a minimizing sequence. By arguing as in the proof of Lemma 6.2 in \cite{BLYZ}, it is possible to find a sequence of origin-symmetric parallelograms $\{P_j\}$, with $\tor ( P _j) = 1$ and orthogonal diagonals, such that $P _j \subseteq L _j \subseteq 2 P _j$. 
If  $\{L _j\}$ is not bounded, we have that also  $\{P_j\}$ is not bounded; moreover, by the assumption $\tor (L _j )= 1$, the monotonicity of torsion by inclusions,  and 
 Saint-Venant inequality \eqref{f:isotor},   the volume of $P _j$ is bounded from below by a positive constant independent of $j$. Then we can apply Lemma 6.1 in \cite{BLYZ} to infer that  the sequence
 $\int_{\mathbb S ^ 1} \log h _ {P _j} \, d V _K$ is not bounded from above. 
Since by Hopf's boundary lemma $|\nabla u _K|  ^ 2$  has a strictly positive minimum on $\partial K$, this implies that  also the sequence $\int_{\mathbb S ^ 1} \log h _ {L _j} \, d \tau_K$ is not bounded from above, contradicting the fact that $\{L _j\}$ is a minimizing sequence. 
Hence, $\{L _j \}$ remains bounded, and by Blaschke's selection theorem we may select a converging subsequence; its limit turns out to be an origin symmetric convex body with unit torsion which solves problem \eqref{f:minpbt0}. 
\end{remark}

 \begin{remark}
In any space dimension, for $K = B$ the logarithmic Minkowski problem 
 \eqref{f:minpbt0}  is solved uniquely by the ball.  Namely, 
 using for brevity the notation $\overline \tau _B$ for the normalized
 cone torsion measure, i.e. $\overline \tau _B:= \tau _ B / \tor (B)$, 
 we have 
 \begin{equation*}
\int_{\mathbb S ^ {n-1}} \log  \Big (\frac{h _L }{h _B} \Big)    d \overline \tau _B \geq  \frac{1}{n+2}  \log  \Big ( \frac{\tor (L) } {\tor (B) } \Big )  \qquad \forall L \in \mathcal K ^ n _*\,.
\end{equation*}
 Indeed, it was proved by Guan-Ni \cite[Proposition 1.1]{GuanNi} that 
 \begin{equation}\label{f:logball0} \int_{\mathbb S^ {n-1}} \log  \Big (\frac{h _L }{h _B} \Big)    d \overline V _B \geq  \frac{1}{n}  \log  \Big ( \frac{ |L| } {|B| } \Big )  \qquad \forall L \in \mathcal K ^ n _*\,,\end{equation}
 with equality if and only if $L = B$, where $\overline V _B:= V _B / |B|$ denotes the normalized cone volume measure of $B$.
 Since $B$ is a ball, we have $ \overline V _B = \overline \tau _B $. Then, by using \eqref{f:logball0} and the Saint-Venant inequality, we obtain
 $$  \int_{\mathbb S ^ {n-1}} \log  \Big (\frac{h _L } {h _B}\Big)    d \overline \tau _B \geq  \frac{1}{n}  \log  \Big ( \frac{ |L| } {|B| } \Big )  \geq   \frac{1}{n+2}  \log  \Big ( \frac {\tor (L) } {\tor (B) }\Big )    \qquad \forall L \in \mathcal K ^ n _*\,.$$
  \end{remark} 

\begin{remark} 
In dimension $n = 2$, the logarithmic Brunn-Minkowski inequality for the first eigenvalue and 
for the torsion
can be easily tested on the class of rectangles. For any $x>0$, set $R_x = ( 0, x) \times (0, 1)$.  Given $\ell _1, \ell _2>0$, and $\lambda \in (0, 1)$, 
we have
$$(1 - \l ) \cdot R_{\ell _1} + _0 \l  \cdot R_{\ell _2}  = R _{\ell} \qquad \text{ with }\  \ell = \ell _ 1 ^ { 1 - \lambda } \ell _ 2 ^ \lambda \,.$$ 
Then, by using the explicit formulas 
$$\begin{array}{ll} 
&\displaystyle \lambda _ 1 ( R _\ell) = \pi ^ 2 \big ( 1 + \frac{1}{\ell ^ 2 }  \big ) 
\\ \noalign{\medskip}
& \displaystyle \tor ( R_\ell) =  \frac{ \ell ^ 3}{12} - \frac{ 16 \ell ^ 4}{\pi ^5} \sum _{ k \geq 0} \frac{e ^ { ( 2k+1) \pi/\ell} -1}{e ^ { ( 2k+1) \pi/\ell } +1  }\frac{1}{(2k+1) ^ 5} \, , 
\end{array}$$ 
the inequalities 
$$ \lambda _ 1 (R_\ell) \leq \lambda _1(R_{\ell_1}) ^ { 1- \l} \lambda _1 (R_{\ell_2}) ^ \l  
 \qquad \text{ and } \qquad 
\tor (R_\ell) \geq \tor(R_{\ell_1}) ^ { 1- \l} \tor(R_{\ell_2}) ^ \l  $$ 
are confirmed respectively via  explicit straightforward computations and via computations done by Mathematica.
\end{remark} 

\medskip


\def\cprime{$'$}

\end{document}